\documentclass[12pt]{amsart}
\usepackage{amsfonts,amsthm,amsmath,amssymb,verbatim}
\usepackage{amscd, color, epsfig, wrapfig}

\newtheorem{thm}{Theorem}[section]
\newtheorem{prop}[thm]{Proposition}
\newtheorem{lemma}[thm]{Lemma}

\newtheorem{cor}[thm]{Corollary}

\theoremstyle{remark}

\newtheorem{example}[thm]{Example}
\newtheorem{remark}[thm]{Remark}
\newtheorem{defin}{Definition}



\def\C{\mathbb{C}}

\def\Z{\mathbb{Z}}
\def\P{\mathbb{P}}

\def\R{\mathbb{R}}

\def\vol{{\rm vol}}

\def\Sym{{\rm Sym}}

\def\d{\partial}

\def\Fc{\mathcal{F}}

\def\Oc{\mathcal{O}}
\def\a{\alpha}
\def\b{\beta}

\def\l{\lambda}

\def\G{\Gamma}

\def\ov{\overline}

\def\w0{\overline{w_0}}

\def\PP{\Delta(P,Q)}
\def\PQ{\widehat P}

\title[Push-pull operators on convex polytopes]
{Push-pull operators on convex polytopes}
\author{Valentina Kiritchenko}
\email{vkiritch@hse.ru}
\thanks{The study has been partially funded by the Russian Academic Excellence
Project '5-100' and by RSF grant 19-11-00056 (Sections 3, 5).}

\address{Laboratory of Algebraic Geometry and Faculty of Mathematics\\
National Research University Higher School of Economics, Russian Federation\\
Usacheva str. 6, 119048 Moscow, Russia}
\address{Institute for Information Transmission Problems, Moscow, Russia}

\date{}

\begin{document}

\begin{abstract}A classical result of Schubert calculus  is an inductive description of
Schubert cycles using divided difference (or push-pull) operators in Chow rings.
We define convex geometric analogs of push-pull operators and describe their applications
to the theory of Newton--Okounkov convex bodies.
Convex geometric push-pull operators yield an inductive construction of
Newton--Okounkov polytopes of Bott--Samelson varieties.
In particular, we construct a Minkowski sum of Feigin--Fourier--Littelmann--Vinberg
polytopes using convex geometric push-pull operators in type $A$.
\end{abstract}

\maketitle
\section{Introduction}
Let $X$ be a smooth algebraic variety, and $E\to X$ a vector bundle of rank two on $X$.
Define the projective line fibration $Y=\mathbb P(E)$ as the variety of all lines
in $E$.
The natural projection $\pi:Y\to  X$ induces the
pull-back $\pi^*:A^*(X)\to A^*(Y)$ and the push-forward $\pi_*:A^*(Y)\to A^{*-1}(X)$ (aka {\em transfer} or {\em Gysin map}) in the (generalized) cohomology rings of $X$ and $Y$.
The {\em push-pull} operator $\pi^*\pi_*:A^*(Y)\to A^{*-1}(Y)$ is a homomorphism of $A^*(X)$-modules, and can be described explicitly via Quillen--Vishik formula
for any algebraic oriented cohomology theory $A^*$ (such as Chow ring, K-theory or algebraic cobordism).
Push-pull operators are used extensively in representation theory (Demazure operators) and in Schubert calculus (divided difference operators).
We discuss convex geometric counterparts of push-pull operators and their applications in the theory of Newton--Okounkov convex bodies and representation theory.

Convex geometric push-pull operators are motivated by the study of Newton-Okounkov polytopes of flag and Bott--Samelson varieties \cite{K18}.
Namely, if $Y=G/B$ is the complete flag variety for a connected reductive group $G$, then there is a natural projective line fibration $\pi_i:Y\to G/P_i$ for every simple root $\a_i$.
Here $P_i\subset G$ denotes the minimal parabolic subgroup associated with $\a_i$,
and $X=G/P_i$ the corresponding partial flag variety.
For instance, if $G=GL_n(\C)$ then points in $G/B$ can be identified with complete flags $(V^1\subset V^2\subset\ldots\subset V^{n-1}\subset \C^n)$ of subspaces, and the map $\pi_i$ forgets the subspace $V^i$.
The corresponding push-pull operator $\d_i:CH^*(Y)\to CH^{*-1}(Y)$ for Chow rings is often called {\em divided difference operator}, while the push-pull operator $D_i:K^*(Y)\to K^{*-1}(Y)$ for the $K$-theory is usually called {\em Demazure operator}.

A classical result of Schubert calculus \cite{BGG,De} is an inductive description of
Schubert cycles $[X_w]\in CH^*(Y)$ for all elements $w\in W$ in the Weyl group of $G$.
Namely, if $w=s_{i_1}s_{i_2}\cdots s_{i_\ell}$ is a reduced decomposition of $w$ into
the product of simple reflections, then
$$[X_w]=\d_{i_\ell}\ldots\d_{i_2}\d_{i_1}[X_{\rm id}]. \eqno(1)$$
A classical result in representation theory \cite{A} is an inductive description of the
Demazure character $\chi_w(\l)$ for every Schubert variety $X_w$ and a dominant weight
$\l$ of $G$:
$$\chi_w(\l)=D_{i_1}D_{i_2}\ldots D_{i_\ell}(e^\l). \eqno(2)$$
While formulas (1) and (2) look similar, there is no direct relation between them since
in (1) operators are applied in the order opposite to that of (2).

In \cite{Ki16}, we defined convex geometric analogs of Demazure operators.
They can be used to construct inductively polytopes $P_\l$ such that the sum of
exponentials over lattice points in $P_\l$ yields the Demazure character $\chi_w(\l)$.
Recently, Naoki Fujita showed that the Nakashima--Zelevinsky polyhedral realizations of
crystal bases for a special reduced decomposition of the longest element $w_0\in W$ can
be constructed inductively using convex geometric Demazure operators in types $A_n$,
$B_n$, $C_n$, $D_n$, and $G_2$ \cite{Fu2}.
In this setting, the convex geometric Demazure operators are applied in the
same order as in (2).

In the present paper, we define different convex geometric analogs of push-pull
operators that are more natural from the perspective of (1) (Section \ref{s.def}).
For a  reduced decomposition $w=s_{i_1}s_{i_2}\cdots s_{i_\ell}$, convex geometric
push-pull operators produce inductively polytopes whose volume polynomials coincide
with the degree polynomials of Bott--Samelson varieties corresponding to collections
of simple roots
$(\a_{i_1})$, $(\a_{i_1},\a_{i_2})$,\ldots , $(\a_{i_1},\a_{i_2},\ldots,\a_{i_\ell})$.
The main tool is the Khovanskii--Pukhlikov ring that can be associated with every
convex polytope (Section \ref{s.reminder}).
We prove an analog of the projective bundle formula for the Khovanskii--Pukhlikov rings of
push-pull polytopes (Section \ref{s.KhP}, Theorem \ref{t.main}) and describe applications in the
theory of Newton--Okounkov convex bodies and in representation theory
(Section \ref{s.applications}).

I am grateful to Evgeny Smirnov and Vladlen Timorin for useful discussions.

\section{Reminder on convex polytopes and Khovanskii--Pukhlikov rings}\label{s.reminder}
In this section, we remind the definition of the polytope ring associated with a convex
polytope $P\subset \R^n$.
This ring was originally introduced by Khovanskii and Pukhlikov in \cite{KhP} to give
a convenient functorial description of the cohomology (or Chow) rings of smooth toric
varieties.
Later, Kaveh used Khovanskii--Pukhlikov ring to give a partial description of cohomology
rings of spherical varieties \cite{Ka}.
Recently, Khovanskii--Pukhlikov rings were applied to cohomology rings of toric bundles
\cite{HKhM}
and Schubert calculus on polyhedral realizations of Demazure crystals \cite{Fu3}.

Let $\Z^n\subset \R^n$ be the integer lattice.
A convex lattice poytope $P\subset \R^n$ is a convex hull of several
points from $\Z^n$.
Recall that two convex polytopes $P$ and $P'$ are called {\em analogous} if they have
the same normal fan.
In particular, there exist linear functions $h_1$,\ldots, $h_d$ on $\R^n$ such that any
polytope $P'$ analogous to $P$ is given by inequalities:
$$h_i(x)\le H_i(P'), \ i=1,\ldots, d$$ for some constants $H_1(P')$,\ldots ,
$H_d(P')\in \R$ that depend on $P'$.
The collection of numbers $(H_1(P'),\ldots, H_d(P'))$ (called support numbers of $P'$)
defines uniquely the polytope $P'$.
If a polytope $P'$ is analogous to $P$ then there is a natural bijection between faces
of $P'$ and faces of $P$.
In the sequel, we denote by $F(P')$ the face of $P'$ that corresponds to a face
$F\subset P$ (in particular, $F=F(P)$).

Denote by $S_P$ the set of all lattice polytopes analogous to $P$.
This set can be endowed with the structure of a commutative semigroup using
{\em Minkowski sum}
$$
P_1+P_2=\{x_1+x_2\in\R^n\ |\ x_1\in P_1,\ x_2\in P_2\}.
$$
We can embed the semigroup of convex polytopes into its Grothendieck group $L_P$,
which is a lattice in $\R^d$.
The elements of $V_P$ are called {\em virtual polytopes} analogous to $P$.
They can be represented by $d$-tuples $(H_1,\ldots, H_d)$ such that
$H_i=H_i(P_1)-H_i(P_2)$ where $P_1$ and $P_2$ are analogous to $P$.
In general, the rank of $L_P$ is smaller than or equal to $d$.
The equality holds if
and only if $P$ is simple, that is, all facets of $P$ can be translated independently
without changing the combinatorial type of $P$.

There is a homogeneous polynomial $vol_P$ of degree $n$ on the lattice $L_P$, called the
{\em volume polynomial}.
It is uniquely characterized by the property that its
value $vol_P(P')$ on any convex polytope $P'\in S_P$ is equal to the volume of $P'$.

The symmetric algebra $\Sym(L_P)$ of $\Lambda_P$ can be
thought of as the ring of differential operators with constant
integer coefficients acting on $\R[L_P]$, the space of all polynomials on $L_P$.
Let $(x_1,\ldots,x_\ell)$ (where $\ell\le d$) be coordinates on $L_P$.
Denote $\frac{\d}{\d x_i}$ by $\d_{x_i}$.
\begin{defin}
The Khovanskii--Pukhlikov ring $R_P$ associated with the polytope $P$ is the quotient
ring
$$\Z[\d_1,\ldots,\d_\ell]/{\rm Ann}(vol_P).$$
Here ${\rm Ann}(vol_P)$ denotes the ideal in $R_P$ that consists of all
differential operators $D$ such that $D(vol_P)=0$.
\end{defin}

\begin{remark}
In what follows, we use that every (virtual) polytope $P'\in L_P$ defines a homogeneous
element $\d_{P'}$ of degree one in $R_P$.
Namely, let $(x_1(P'),\ldots, x_\ell(P'))$ be coordinates of $P'$.
Put $\d_{P'}=x_1(P')\d_1+\ldots+x_\ell(P')\d_\ell.$
\end{remark}

\section{Definition of push-pull polytopes}\label{s.def}
Let $P$, $Q\subset\R^n$ be convex polytopes such that $Q$ is analogous to a
{\em codimension two truncation} of $P$.
More precisely, let $\Fc=$ \{$F_1$, \ldots, $F_k$\} be a (possibly empty) collection of
codimension two faces of $P$.
Let $\PQ$ be a polytope analogous to $P$.
For every $i=1$,\ldots, $k$, choose a linear function $\psi_i(x)$ on $\R^n$ such that
$\psi_i$ takes a constant value $\Psi_i(\PQ)$ on $F_i(\PQ)$, and $\psi_i(x)< \Psi_i(\PQ)$
for all  $x\in\PQ\setminus F_i(\PQ)$.
Fix a constant $\Psi_i(Q)< \Psi_i(\PQ)$.
A {\em codimension two truncation} $Q$ of $\PQ$ is obtained from $\PQ$ by intersecting
$\PQ$ with half-spaces $\{\psi_i(x)\le \Psi_i(Q)\}$ for $i=1$,\ldots, $k$.
In particular, $Q$ has $k$ extra facets $\G_i=\{\psi_i(x)=\Psi_i(Q)\}\cap \PQ$ for
$i=1$,\ldots, $k$ (in addition to facets that come from $\PQ$).
We assume that $\Psi_i(Q)$ is sufficiently close to $\Psi_i(\PQ)$ so that $\PQ\setminus Q$
does not contain any vertices of $\PQ$ except for those in $\bigcup_{F_i\in\Fc} F_i$.

In what follows, we will need the following description of the volume polynomial of $Q$.
Let $s$ be a positive real number.
Define the nonconvex polytope $\ov Q(s)\subset\R^n$ as the union of convex polytopes
$$\bigcup_{i=1}^k \{\psi_i(x)\ge \Psi_i(\PQ)+s(\Psi_i(Q)-\Psi_i(\PQ))\}\cap \PQ.$$
In particular, $\vol (\PQ)=\vol(Q)+\vol (\ov Q(1))$, as $\PQ=Q\cup \ov Q(1)$, and
$Q\cap \ov Q(1)=\bigcup_{i=1}^k \G_i$.
In general,
$$\vol(\ov Q(s))=\left(\sum_{i=1}^{k} a_i\vol_{n-2}(F_i(\PQ))\right)
\frac{s^2}{2}+\sum_{j=3}^{n}g_j(\PQ)s^j.$$
where $a_1,\ldots,a_k\in\R$ and $g_j(\PQ)$ is a linear combination of volumes of certain
codimension $j$ faces of $\PQ$.
This implies the following formula for the volume of $Q$:
$$\vol(Q)=\vol(\PQ)-\left(\sum_{i=1}^{k} a_i\vol_{n-2}(F_i(\PQ))\right) \frac{s^2}{2}-
\sum_{j=3}^{n}g_j(\PQ)s^j.
\eqno(*)$$

\begin{defin} The {\em push-pull polytope} $\PP\subset\R^{n+1}=\R^n \times \R$
is the {\em Cayley sum} of polytopes $P$ and $Q$.
More precisely, $\PP$ is the convex hull of the following set:
$$(P\times 1) \cup (Q\times 0).$$
\end{defin}
This definition is slightly different from the usual definition of the Cayley sum of polytopes.
The reason for this modification will be seen from Examples \ref{e.FFLV0} and \ref{e.FFLV} below.

In what follows, we mostly consider not the particular polytope $\PP$ but the whole family of polytopes analogous to $\PP$.
It is easy to check that if $P'$ is analogous to $P$, then $\PP$ is analogous to $\Delta(P',P'-P+Q)$.
In particular, the space of all virtual polytopes analogous to $\PP$ depends only on the difference $Q-P$.

Below we consider motivating examples of push-pull polytopes.

\subsection{Minkowski sum with a segment}
A special codimension two truncation of $P$ can be obtained using the Minkowski sum with a segment.
Let $I\subset\R^n$ be a segment, and $P'$ a polytope analogous to $P$.
It is easy to show that the Minkowski sum $Q=P'+I$ is a codimension two truncation of a polytope $\PQ$ analogous to $P$.
Indeed, the facets of $P'+I$ are either parallel to facets of $P'$ or equal to Minkowski sums $F+I$ where $F\subset P'$ is a face of codimension two.
Then $\PP$ is a convex hull of
$$(P\times 1) \cup ((P'+I)\times0).$$

\begin{example}\label{e.FFLV0}
Let $n=2$, and $P\subset\R^2$ the triangle with the vertices $(0,0)$, $(1,0)$, $(0,1)$.
Take $I=[(0,0),(0,1)]$, and $P'=P$.
Then $Q=P+I$ is analogous to trapezoid obtained from $\PQ=2P$ by truncating the vertex $(2,0)$ (=codimension two face).
Then the push-pull polytope $\PP\subset\R^3$ is the FFLV polytope in type $A_2$ corresponding to the weight $\rho$
(see \cite{FFL} for the definition of FFLV polytopes in type $A$ and their representation-theoretic meaning).
The figure below shows $P$, $Q$ and $\PP$, respectively (from left to right).
\medskip

\includegraphics[height=3.5cm]{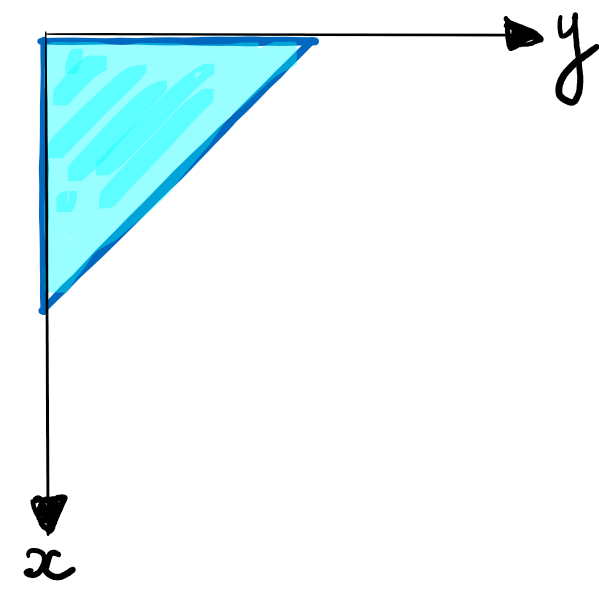} \quad
\includegraphics[height=3.4cm]{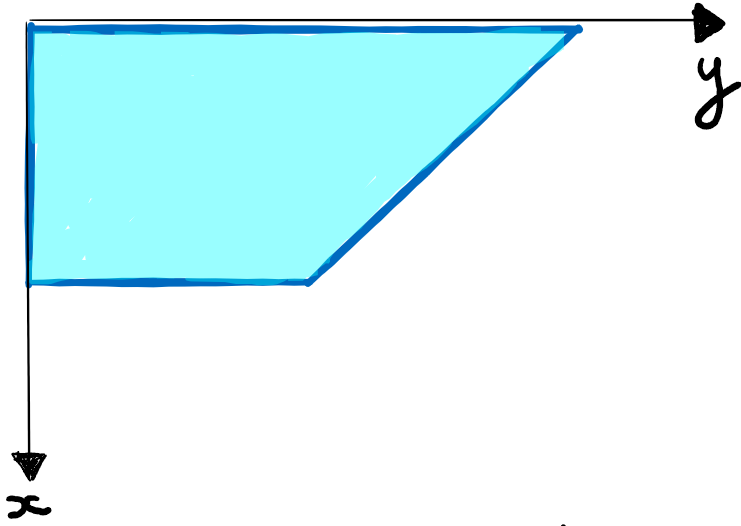} \quad
\includegraphics[height=3.5cm]{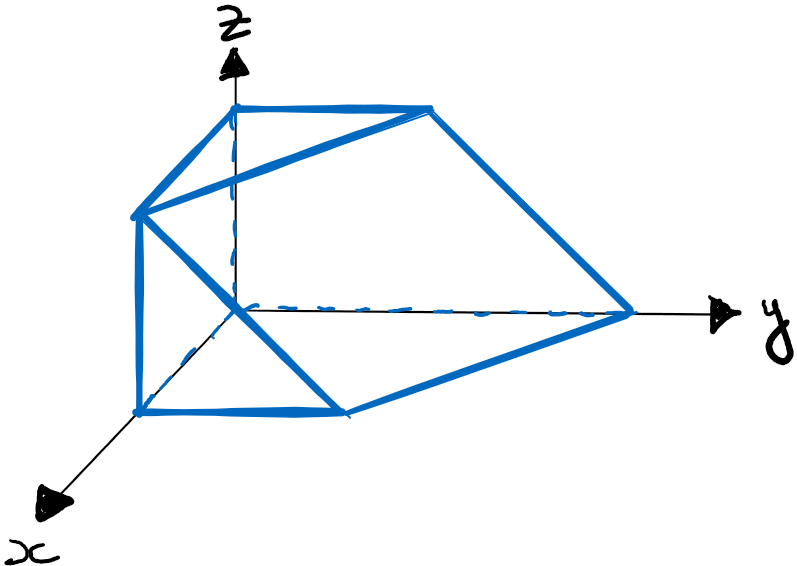}
\end{example}

\begin{example}{cf. \cite[Section 6.4]{An}} \label{e.FFLV}
Let $n=2$, and $P\subset\R^2$ the trapezoid with the vertices $(0,0)$, $(2,0)$, $(0,1)$, $(1,1)$.
Take $I=[(0,0),(0,1)]$, and $P'=P$.
Then $Q=P+I$ is analogous to a 5-gon obtained from $\PQ$ by truncating the vertex $(3,0)$.
Here $\PQ$ is the trapezoid with the vertices $(0,0)$, $(3,0)$, $(0,2)$, $(1,2)$.
Then the push-pull polytope $\PP\subset\R^3$ is the Minkowski sum of the FFLV polytope from Example \ref{e.FFLV0} and the segment $J=[(0,0,0),(1,0,0)]$.
By shrinking $J$ we get a degeneration to the FFLV polytope.
The figure below shows $P$, $Q$ and $\PP$, respectively (from left to right).
\medskip

\includegraphics[height=4.5cm]{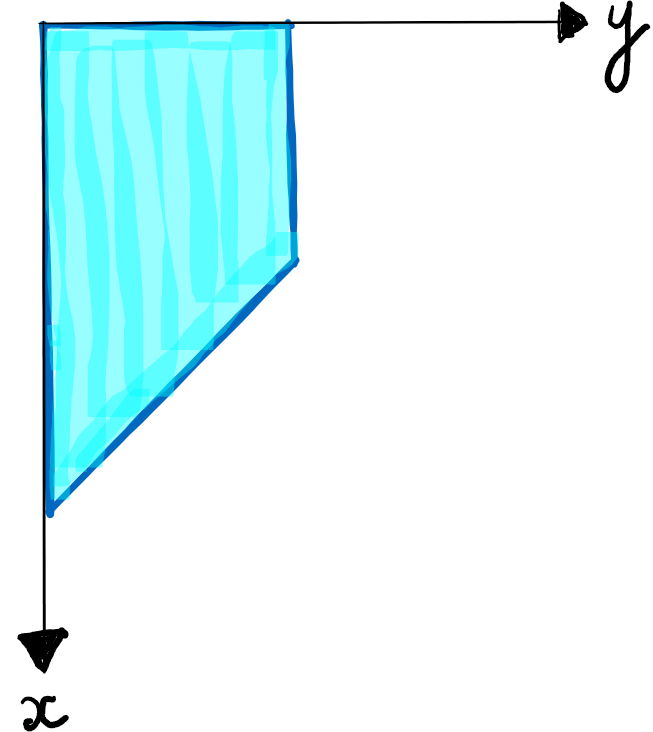} \quad \includegraphics[height=4.5cm]{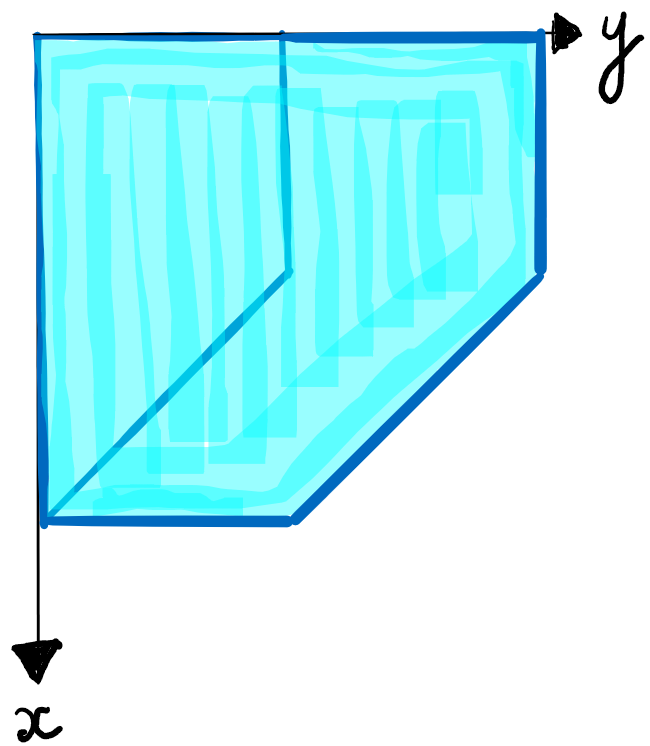} \quad
\includegraphics[height=4.5cm]{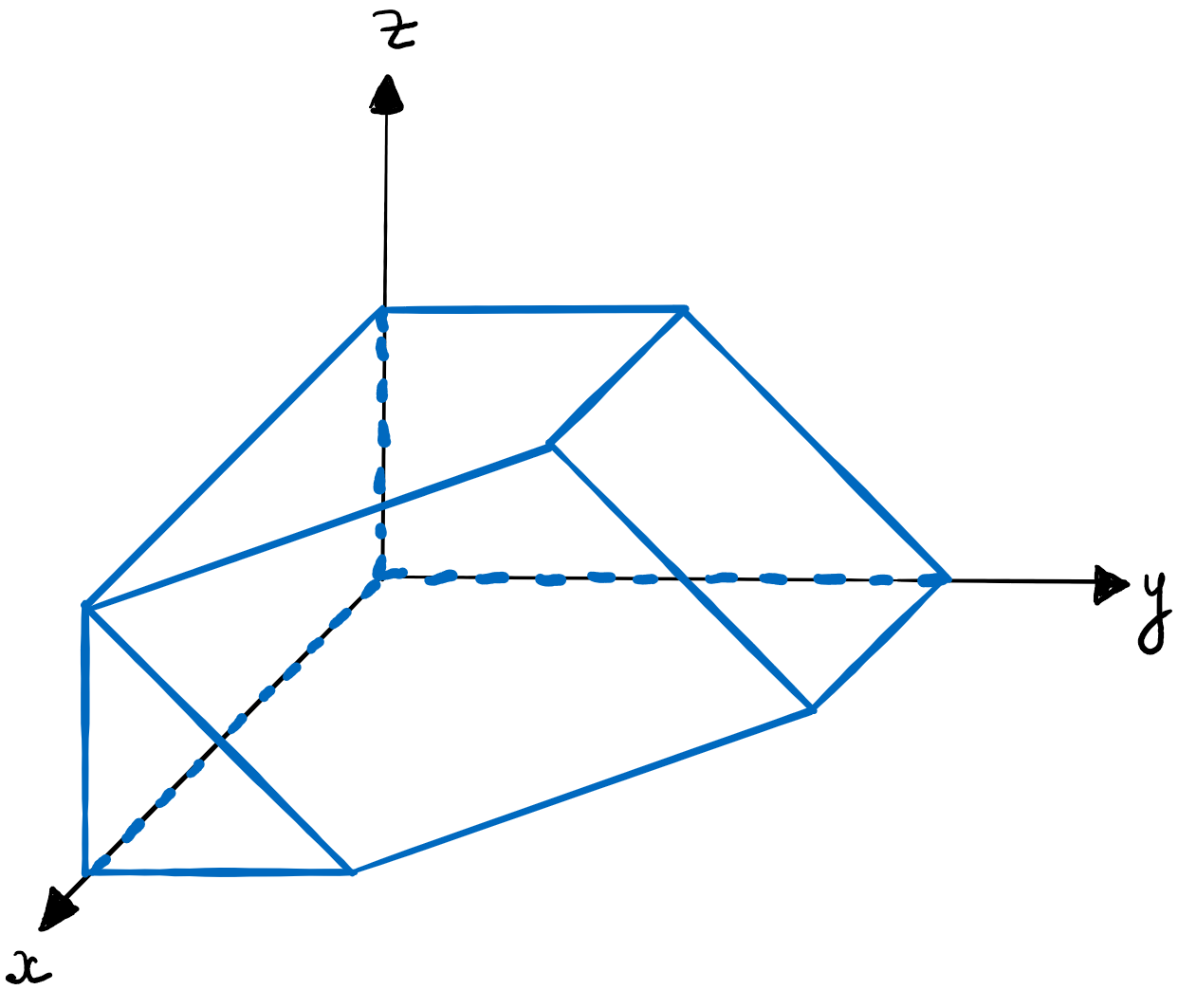}

\end{example}

\subsection{Analogous $P$ and $Q$}
When $Q=\PQ$ is analogous to $P$ (that is, no codimension two faces are truncated),
the push-pull polytope $\PP$ coincides with the Newton polytope of a projective line
fibration $Y=\mathbb P^1(E)$ over the projective toric variety $X$ corresponding to the
polytope $P$.
We assume that $E=\mathcal O\oplus \mathcal L$ is split so $Y$ is also a toric variety.
In this case, there is a simple relation between the polytope $Q$ and the line bundle $\mathcal L$.
Namely, identify the Picard group of $X$ with the Grothendieck group of lattice polytopes analogous to $P$.
Denote by $[\mathcal L]$ the  (possibly virtual) polytope corresponding to $\mathcal L$.
Then $Q=P+[\mathcal L]$.
In terms of Khovanskii-Pukhlikov rings, this can rephrased as follows.
Under the isomorphism $CH^*(X)\simeq R_P$ the first Chern class $c_1(\mathcal L)$
gets mapped to $\d_{Q-P}$.
In particular, the construction of Grossberg--Karshon cubes \cite{GK} motivated by Demazure character formula (2) for Bott--Samelson varieties can also be reproduced in the spirit of (1) using convex-geometric push-pull operators.
We discuss this approach in more detail in Section \ref{ss.GK}.

While Grossberg--Karshon cubes can be realized as Newton--Okounkov polytopes of Bott--Samelson varieties for some line bundles \cite{Fu,HY,HY17},
they do not degenerate to Newton--Okounkov polytopes of flag varieties (instead they turn into twisted cubes, which are not true polytopes).
More general push-pull operators (with codimension two truncation) allow us to produce polytopes that do not break in the limit when passing to flag varieties.
For instance, Example \ref{e.FFLV} produces the Newton--Okounkov polytope of a Bott--Samelson variety in type $A_2$ with the desired degeneration.
An example in type $A_3$ will be considered in Section \ref{ss.BS}

\section{Khovanskii--Pukhlikov  ring of a push-pull polytope}\label{s.KhP}
In this section,
we show that the Khovanskii--Pukhlikov rings of $P$ and of $\Delta:=\PP$ are related in the same manner as cohomology rings $H^*(X)$ and $H^*(Y)$ of smooth varieties $X$ and $Y$, whenever $Y$ is a projectivization of a rank two vector bundle over $X$.
We use notation of Sections \ref{s.reminder} and \ref{s.def}.

Let $R_P$ and $R_\Delta$ denote the Khovanskii--Pukhlikov rings of polytopes $P$ and $\Delta$, respectively.
More precisely, suppose that $R_P$ is defined using the lattice $L_P$, and $L_P$ is spanned by polytopes $P_1$,\ldots, $P_\ell$.
Define $L_\Delta$ as the lattice spanned by $\Delta(P+P_i,Q+P_i)$ for $i=1$,\ldots, $\ell$, and by an extra generator $C$.
The virtual polytope $C$ can be informally thought of as the cone over $Q-P$.
To give a formal definition we will need the following description of polytopes analogous to $\Delta$.

\begin{defin} Let $P'$ be a polytope analogous to $P$, and $s$ a positive real number.
Suppose that $P'-P=s_1P_1+\ldots+s_\ell P_\ell$, where all $s_i$ are nonnegative.
Define the polytope $\Delta(s,P')\subset\R^n\times \R$ as
the convex hull of
$$(P'\times s) \cup \left((P'+s(Q-P))\times 0\right).$$
\end{defin}

In particular, $\Delta(1,P)=\Delta$.
It is easy to check that the polytopes $\Delta(s,P')$ and $\Delta$ are analogous, and that the following identity holds:
$$\Delta(s,\sum_{i=1}^\ell s_iP_i)+\Delta(t,\sum_{i=1}^\ell t_iP_i)=\Delta(s+t,\sum_{i=1}^\ell (s_i+t_i)P_i).$$
Let $L_\Delta$ be the Grothendieck group of the semigroup of polytopes $\Delta(s,P')$.
Then the lattice $L_\Delta$ has a unique basis $C$, $\Delta_1$,\ldots, $\Delta_\ell$,
such that the polytope $\Delta(s,P')$ in this basis has coordinates
$(s,s_1,\ldots, s_\ell)$.
In particular, $C$ can be defined as the formal difference of polytopes $\Delta(2,P)$
and $\Delta$.

Consider the monomorphism of lattices $L_P\hookrightarrow L_\Delta$, which sends $P_i$ to $\Delta_i$ for $i=1$,\ldots, $\ell$.
In what follows, we identify $L_P$ with the sublattice of $L_\Delta$ using this
monomorphism.
Define the polynomial function $\vol_\Fc$ on $L_P$ by the condition
$$\vol_\Fc(P')=\sum_{i=1}^{k}a_i\vol_{n-2}(F_i(P'))$$
for all $P'\in L_P$ that are analogous to $P$.
Similarly, we define polynomial functions $g_j$ on $L_P$ for $j=3$,\ldots, $n$ so that
formula $(*)$ from Section
\ref{s.def} holds for all $P'\in L_P$.
Define the polynomial $q(s,P')$ on $\R\times L_P$ by the formula:
$$q(s,P')=\vol_\Fc(P')\frac{s^2}{2}+\sum_{j=3}^{n} g_j(P')s^j.$$
\begin{thm}\label{t.main}
Suppose that there exist a homogeneous element $D_{\Fc}\in R_P$ of degree two and
homogeneous elements $D_j\in R_P$ of degrees $3$,\ldots, $n$
such that
$D_{\Fc}(\vol_P)=\vol_\Fc$ and $D_j(\vol_P)=g_j$.
Put
$$c_1=\d_{\PQ}-\d_P, \quad c_2=D_{\Fc}.$$
If $q(s,P')$ satisfies the condition:
$$\left(\frac{\d^2}{\d s^2}-c_1\frac{\d}{\d s}+c_2\right)q(s,P'+sc_1)=\vol_\Fc(P'+sc_1),
\eqno(**)$$
then there is a ring isomorphism
$$R_\Delta\simeq R_P[x]/(x^2-c_1x+c_2).$$
In particular, $c_1$ and $c_2$ can be thought of as analogs of the Chern classes of a rank two vector
bundle.
\end{thm}
\begin{remark}
Note that if $P$ is simple, then there always exists a homogeneous element
$D_{\Fc}\in R_P$ of degree two such that $D_{\Fc} (\vol_P)=\vol_\Fc$.
Indeed, every face $F\subset P$ of codimension two is the intersection of two facets
$\G_1$ and $\G_2$.
Since $P$ is simple, every facet $\G\subset P$ defines a homogeneous element
$\d_\G\in R_P$ of degree one.
Put $D=\d_{\G_1}\d_{\G_2}$.
It is easy to check that $D(\vol_P)=a\vol_F$ for some constant $a\in\R$ that depends
on the polytope $P$.
Similarly, one can show that $D_3$,\ldots, $D_n$ also exist for simple $P$.

However, if $P$ is not simple then the existence of $D_{\Fc}$ and $D_3$,\ldots, $D_n$
can not be taken for granted.
\end{remark}
\begin{remark} Condition $(**)$ is void if $\Fc=\varnothing$.
Also, it holds trivially in the case where $g_3=\ldots=g_n=0$
(that is, $q(s,P')$ is a quadratic polynomial in $s$), and $c_1c_2=0$ in $R_P$.
This is the case in Examples \ref{e.FFLV0} and \ref{e.FFLV} for dimension reasons.

From a geometric viewpoint, condition $(**)$ is ``local'', that is, depends only on the
structure of the polytope $P$ in a sufficiently small neighborhood of the union of faces
$\bigcup_{F_i\in\Fc}F_i$.
Because of this, it is sometimes easier to verify $(**)$ than more ``global'' statements
about $P$ and the volume polynomial $\vol_P$.
\end{remark}
\begin{proof}
First, we construct a ring monomorphism $\varphi: R_P\hookrightarrow R_\Delta$.
Put $\d_{s_i}=\frac{\d}{\d s_i}$ for $i=1$,\ldots, $\ell$, and $\d_s=\frac{\d}{\d s}$.
Let $D$ be a polynomial in $\d_{s_1}$,\ldots, $\d_{s_\ell}$.
Since $L_P\subset L_\Delta$, the differential operator $D$ can also be regarded as
the polynomial $\varphi(D)$ in $\d_s$, $\d_{s_1}$,\ldots, $\d_{s_\ell}$.
It remains to check that the map $\varphi$ is well-defined, that is, $D(\vol_P)=0$
implies $\varphi(D)(\vol_{\Delta})=0$.
By definition of $\Delta(s,P')$ and formula $(*)$ for the volume of $Q$ from Section
\ref{s.def}, we have the following explicit formula for the volume polynomial
$\vol_{\Delta}$:
$$\vol_\Delta(\Delta(s,P'))=
\int_{0}^{s}\vol_P(x(t))-q(t,x(t))dt,$$
where $x(t)=P'+t(\PQ-P)$.
Since $D(\vol_P)=0$ and $\vol_\Fc=D_{\Fc}(\vol_P)$ we have that
$D(\vol_\Fc)=D_{\Fc}D(\vol_P)=0$.
Similarly, we have that $D(g_i)=D_iD(\vol_P)=0$.
Hence, $\varphi(D)(\vol_\Delta)=0$.

Second, we show that
$$(\d_s-\d_{\PQ-P})\d_s\vol_\Delta=-D_{\Fc}\vol_\Delta.$$
Using the explicit formula for $\vol_\Delta(\Delta(s,P'))$ we get
$$(\d_s-\d_{\PQ-P})\d_s\vol_\Delta=(\d_s-\d_{\PQ-P})
\left(\vol_P(x(s))-q(s,x(s))\right).$$
By definition of the directional derivative $\d_{\PQ-P}$ we have that
$(\d_s-\d_{\PQ-P})\vol_P(x(s))=0$.
Hence,
$$(\d_s-\d_{\PQ-P})
\left(\vol_P(x(s))-q(s,x(s))\right)=-(\d_s-\d_{\PQ-P})q(s,x(s)).$$
By condition $(**)$ the right hand side is equal to $-D_{\Fc}\vol_\Delta (\Delta(s,P'))$.
Indeed,
$$D_{\Fc}\vol_\Delta (\Delta(s,P'))=\int_{0}^{s}D_{\Fc}\vol_P(x(t))-D_{\Fc}q(t,x(t))dt=$$
$$=\int_{0}^{s}\vol_{\Fc}(x(t))-D_{\Fc}q(t,x(t))dt=
\int_{0}^{s}\d_t(\d_t-\d_{\PQ-P})q(t,x(t))dt.$$

Finally, we establish the isomorphism $\Phi: R_P[x]/(x^2-c_1x+c_2)\simeq R_\Delta$ by
sending $x$ to $\d_s$.
%and $R_P$ to $\varphi(R_P)\subset R_\Delta$.
This yields a well-defined ring homomorphism
$$\Phi:g+hx\mapsto \varphi(g)+\varphi(h)\d_s,$$
since we already checked that $\d_s^2-c_1\d_s+c_2=0$ in $R_\Delta$.
Clearly, $\Phi$ is surjective since $R_\Delta$ is generated by $\d_s$ and $\varphi(R_P)$.
It is injective since $\d_s\notin\varphi(R_P)$.
\end{proof}
As a byproduct of the proof of Theorem \ref{t.main} we get the following
\begin{cor} \label{c.volume}
The volume polynomial $\vol_\Delta$ of the push-pull polytope satisfies the differential  equation:
$$F''-(\d_{\PQ}-\d_P)F'+D_{\Fc} F=0,$$
where an unknown function $F$ is a function of two variables $s\in \R$ and
$x\in L_P\otimes \R=\R^\ell$, and $F'$ is the derivative of $F$ with respect to $s$.
\end{cor}
\begin{example}\label{e.R_121}
We continue Example \ref{e.FFLV}.
A polytope $P'$ is analogous to $P$ if $P'$ is given by inequalities
$$0\le (x-x_0); \quad 0\le (y-y_0)\le b; \quad (x-x_0)+(y-y_0)\le a+b.$$
Choose a basis $P_1$, $P_2$, $P_3$, $P_4$ in $L_P$ so that $P'$ has coordinates $(a,b,x_0,y_0)$ in this basis.
In particular, $\d_{s_1}=\frac{\d}{\d a}:=\d_a$ and $\d_{s_2}=\frac{\d}{\d b}:=\d_b$.
Since the area of $P'$ does not change under parallel translations we have that
$\d_{s_3}\vol_P=\d_{s_4}\vol_P=0$.
In other words, the volume polynomial $\vol_P=ab+\frac{b^2}{2}$ does not depend on $x_0$ and $y_0$.
It is easy to check that the differential operators $(\d_b^2-\d_a\d_b)$ and $\d_a^2$ annihilate $\vol_P$.
Hence, the Khovanskii--Pukhlikov ring $R_P$ is isomorphic to the quotient of the polynomial ring $\Z[\d_a,\d_b]$ by the ideal $(\d_a^2, \d_b^2-\d_a\d_b)$.

We now compute $c_1$ and $c_2$.
We have that $\PQ-P=P_2$, hence, $\d_{\PQ-P}=\d_b$.
Since $\Fc$ consists of a single vertex of $\PQ$, and we may choose any $D_{\Fc}$ such that
$D_{\Fc}\vol_P=1$.
For instance, take $D_{\Fc}=\d_a\d_b$.
Hence, the statement of Theorem \ref{t.main} reduces to the following isomorphisms:
$$R_\Delta\simeq R_P[x]/(x^2-\d_b x+\d_a\d_b)\simeq \Z[\d_s,\d_a,\d_b]/(\d_a^2, \d_b^2-\d_a\d_b, \d_s^2-\d_b\d_s+\d_a\d_b),$$
which can be easily checked by direct calculation.
Indeed, the volume polynomial $\vol_\Delta$ is equal to $s(ab+\frac{b^2}{2})+\frac{s^2}{2}(a+b)$.
Clearly, the differential operator $\d_s^2-\d_b\d_s+\d_a\d_b$ annihilates $\vol_\Delta$.

%Note that we could have used the relations
%$$\d_a^2=\d_b^2-\d_a\d_b=\d_s^2-\d_b\d_s+\d_a\d_b=0$$
%in the ring $R_\Delta$ in order to compute the volume polynomial
%$\vol_\Delta$.
%Namely, $\vol_\Delta$
\end{example}

\section{Applications to Schubert calculus and representation theory}\label{s.applications}

\subsection{Formula for the Schubert cycles and Demazure character formula} \label{ss.GK}
We now explain a simple convex geometric relation between formula (1) (for the Schubert cycles) and formula (2) (for the Demazure characters).
Namely, we construct Grossberg--Karshon cubes for a given reduced sequence $(\a_{i_1},\ldots, \a_{i_\ell})$ of simple roots by two different methods.
Both constructions are inductive but the first method uses terminal subwords of the word $I=(i_1, i_2,\ldots, i_\ell)$ while the second method uses initial subwords of $I$.

The first method is the same as in \cite{GK}, and can also be described within a more general framework of convex geometric Demazure operators \cite[Section 4]{Ki16}.
Namely, we consider a (possibly virtual) polytope $P\subset\R^n$ obtained as $$P=D_{i_1}D_{i_2}\ldots D_{i_\ell}(p),$$
where $p\in\R^n$ and $D_1$,\ldots, $D_\ell$ are convex geometric Demazure operators for $G$ (see \cite[Section 4.3]{Ki16} for more details).
In this case, the intermediate polytopes $D_{\ell}(p)$, $D_{\ell-1}D_{\ell}(p)$,\ldots , and $D_2\cdots D_{\ell-1}D_{\ell}(p)$ are (possibly twisted) Grossberg-Karshon cubes for terminal subwords $(i_\ell)$, $(i_{\ell-1},i_\ell)$,\ldots , $(i_2,\ldots , i_{\ell-1}, i_{\ell})$, respectively.

The second method uses push-pull polytopes.
Recall that the Bott-Samelson variety $R_I$ can be constructed as a tower of projective line fibrations:
$$
\{{\rm pt}\}=R_\varnothing\leftarrow R_{(i_1)}\leftarrow R_{(i_1,i_2)}\leftarrow\ldots\leftarrow R_{(i_1,i_2,\ldots , i_{\ell-1})}\leftarrow R_I.$$
This tower is used in one of the proofs of formula (1) (see \cite{De} or \cite[Theorem 3.6.18]{Ma}).
Every variety in this tower is obtained form the previous variety as the projectivization of a rank two vector bundle.
Using push-pull polytopes we will turn a tower of Bott--Samelson varieties into a tower of convex polytopes:
$$\{p\}=P_{\varnothing}\leftarrow P_{(i_1)}\leftarrow P_{(i_1,i_2)}\leftarrow\ldots\leftarrow P_{(i_1,i_2,\ldots ,i_{\ell-1})}\leftarrow P_I.$$
We show that the intermediate polytopes in this tower can be identified with suitable Grossberg-Karshon cubes for initial subwords $(i_1)$, $(i_1,i_2)$,\ldots , and $(i_1,i_2,\ldots , i_{\ell-1})$, respectively.
In particular, we will show that the resulting polytope $P_I$
is analogous to $P$.
This observation exhibits a relation beween formulas (1) and (2).

We now describe construction of $P_I$ in more detail.
We first construct $P_I$ in convex geometric terms and then explain the relation to Bott--Samelson varieties and Grossberg--Karshon cubes.
Let ($\b_1$,\ldots, $\b_\ell$) be a collection of vectors in $\R^\ell$ (it is possible that $\b_i=\b_j$ for $i\ne j$).
Fix a positive definite inner product $\langle\cdot,\cdot\rangle$ on $\R^{\ell}$, and define the linear function $(\cdot,\b_i)$ on $\R^{\ell}$ by the formula:
$$(\a,\b_i)=2\frac{\langle\a,\b_i\rangle}{\langle \b_i,\b_i\rangle}.$$
Denote by $s_{\b_i}$ the reflection through the hyperplane orthogonal to $\b_i$, i.e.,
$$s_{\b_i}:\R^{\ell}\to\R^{\ell}; \quad s_{\b_i}:\a\mapsto \a-(\a,\b_i)\b_i.$$
Let $(x_1,\ldots,x_\ell)$ be coordinates on $\R^\ell$.
Define $(\ell-1)$ linear functions on $\R^\ell$:
$$f_1=(\b_1,\b_2)x_1; \quad f_2=(\b_1,\b_3)x_1+(\b_2,\b_3)x_2; \ \ldots \ ;$$
$$f_{\ell-1}=(\b_1, \b_\ell)x_1+(\b_2,\b_\ell)x_2+\ldots+(\b_{\ell-1},\b_\ell)x_{\ell-1}.$$
In particular, $f_i$ depends only on $x_1$,\ldots, $x_i$.
We now associate with a vector $\l\in\R^n$ the polytope $P_I(\l)$ defined by inequalities (cf. \cite[(2.21)]{GK}):
$$0\le x_1\le (\l,\b_1); \quad 0\le x_2\le -f_1(x_1)+(\l,\b_2);
\quad 0\le x_3\le -f_2(x_1,x_2)+(\l,\b_3);\ \ldots  \ ;$$
$$0\le x_\ell\le -f_{\ell-1}(x_1,\ldots,x_{\ell-1})+(\l,\b_{\ell}).$$
The polytope $P_I(\l)$ is a combinatorial cube and can be thought of as a (possibly virtual) multidimensional version of a trapezoid.
There is a unique vertex $p_I(\l)$ such that all coordinates of $p_I(\l)$ are nonzero (for generic $\l$).
It is easy to find these coordinates:
$$p_I(\l)=((\l,\b_1), (\l,s_{\b_1}(\b_2)), \ldots, (\l,s_{\b_1}\cdots s_{\b_{\ell-1}}(\b_\ell))). \eqno(3)$$
Note that $p_I(\l)=p(\mu)$ if and only if $P_I(\l)=P_I(\mu)$.

The polytope $P_I(\l)$ can be constructed inductively using push-pull polytopes.
The induction step is based on the following
\begin{lemma}\label{l.Dem} Let $I^1$ denote the sequence $(\b_2,\ldots, \b_\ell)$.
If $P_{I^1}(\l)$ and $P_{I^1}(\l-\b_1)$ are true (not virtual) polytopes then
$P_I(\l)$ is analogous to the push-pull polytope $\Delta(P_{I^1}(\l-\b_1),P_{I^1}(\l))$.
\end{lemma}
\begin{proof}
Consider two parallel facets $\G_0:=P_I(\l)\cap \{x_1=0\}$ and $\G_1:=P_I(\l)\cap\{x_1=(\l,\b_1)\}$ of $P_I(\l)$.
Then $P_I(\l)=\Delta(\G_1,\G_0)$.
By construction $\G_0=P_{I^1}(\l)\subset\R^{\ell-1}:=\R^{\ell}\cap \{x_1=0\}$.
It is easy to check that if we identify $\G_1$ with its projection to $\R^{\ell-1}$ along the $x_1$-axis, then $\G_1=P_{I^1}(s_{\b_1}(\l))$.
Indeed, the vertex $p_I(\l)$ of $\G_1$ gets projected to the point
$$((\l,s_{\b_1}(\b_2)), \ldots, (\l,s_{\b_1}\cdots s_{\b_{\ell-1}}(\b_\ell)))\in\R^{\ell-1}.$$
It remains to use that $(\l,\mu)=(s_{\beta_1}(\l),s_{\beta_1}(\mu))$ for all $\l$, $\mu\in\R^{\ell}$.

Since $s_{\beta_1}(\l)=\l-(\l,\b_1)\b_1$, we have $\G_1=\G_0-(\l,\b_1)P_{I^1}(\b_1)$.
By construction $P_I(\l)$ is the convex hull of $\G_0$ and $\G_1$, hence, $P_I(\l)$ is analogous to $\Delta(P_{I^1}(\l-\b_1),P_{I^1}(\l))$.
Here we use that $(\l,\b_1)>0$ since $P_I(\l)$ is not virtual.
\end{proof}

We now consider the case of Bott--Samelson varieties.
Let $(\b_1,\ldots,\b_\ell)=(\a_{i_\ell},\ldots,\a_{i_1})$, and $\l$ a weight of $G$.
In what follows, $L_\l$ denotes the line bundle on $G/B$ associated with a weight $\l$,
and $p_I:R_I\to G/B$ denotes the projection of the Bott--Samelson variety to
the flag variety.
By construction $R_{I}=\P(E)$ where $E$ is a rank two vector bundle over $R_{I^\ell}$.
The bundle $E$ can be chosen so that there is a short exact sequence
$0\to \Oc_{R_{I^\ell}}\to E\to p_{I^\ell}^*L_{\a_{i_\ell}}\to 0$.
In particular, $c_1(E)=c_1(p_{I^\ell}^*L_{\a_{i_\ell}})$ and $c_2(E)=0$.
Hence, if the Khovanskii--Pukhlikov ring $R_{P_{I^\ell}}$ is isomorphic
to $CH^*(R_{I^\ell})$ then Theorem \ref{t.main} together with Lemma \ref{l.Dem}
imply the ring isomorphism $R_{P_{I}}\simeq H^*(R_I)$.

Note that in order to apply Lemma \ref{l.Dem} we might have to choose a true polytope
$P_{I^1}$ analogous to $P_{I^1}(\l)$ so that $P_{I^1}-P_{I^1}(\beta_1)$ is also
true.
In particular, $P_I(\l)$ itself is often a virtual
(twisted) polytope.
In other words, we replace the (not very ample) line bundle $p_I^*L_\l$ on $R_I$ by a very
ample line bundle $L$ such that $L\otimes L(-\beta_i)$ is also very ample.

Hence, the inductive construction of the polytope $P_I$ via a sequence of push-pull
polytopes repeats the inductive construction of the Bott--Samelson variety $R_I$.
On the other hand, the definition of $P_I(\l)$  by explicit inequalities together
with calculations in \cite[3.6]{GK} exhibit $P_I(\l)$ as the Karshon--Grossberg cube
for $R_I$ and $\l$.

\begin{remark} In the case of Bott--Samelson varieities, coordinates of $p_I(\l)$ computed in (3) are exactly the coefficients in the Chevalley--Pieri formula in the Chow ring $CH^*(R_I)$.
Recall that the Picard group of $R_I$ is spanned by the classes of hypersurfaces $R_{I^j}$ where $I^j$ denotes the sequence $(i_1,\ldots,i_{j-1},i_{j+1},\ldots,i_{\ell})$.
The Chevalley--Pieri formula for Bott--Samelson varieties yields the following decomposition of $c_1(p_I^*L_\l)$ in the basis $([R_{I^j}])_{1\le j\le \ell}$:
$$c_1(p_I^*L_\l)=\sum_{j=1}^\ell (\l,s_{\b_1}\cdots s_{\b_{j-1}}(\b_j))[R_{I^j}].$$
In particular, the induction step of Lemma \ref{l.Dem} in this setting can be viewed as a convex geometric counterpart of \cite[Proposition 2.1]{De}.

On the other hand, there is a simple convex geometric analog of Chevalley--Pieri formula in the Khovanskii--Pukhlikov ring of $P_I$.
Let $\G_j$ denote the facet of $P_I$ given by equation $x_j=0$.
The polytope $P_I$ is simple so $\G_j$ defines an element $\d_{\G_j}\in R_{P_I}$.
Then (3) implies immediately the following identity:
$$\d_{P_I(\l)}= \sum_{j=1}^\ell (\l,s_{\b_1}\cdots s_{\b_{j-1}}(\b_j))\d_{\G_j}.$$
\end{remark}

\subsection{Minkowski sums of FFLV polytopes} \label{ss.BS}
Construction of convex geometric push-pull operators is mainly motivated by study of Newton--Okounkov convex bodies of Bott-Samelson varieties.
Namely, polytopes that arise as conjectural Newton--Okounkov polytopes can sometimes be also constructed by iterated application of push-pull operators.
In such cases, properties of Khovanskii--Pukhlikov rings of push-pull polytopes (Theorem \ref{t.main}) allow us to prove that we indeed constructed the full Newton--Okounkov polytope (the same kind of argument was used in Section \ref{ss.GK}).
Below we illustrate this method in the case of Bott--Samelson variety $R_{(1,2,1,3,2)}$ in type $A_3$.
This is a partial case of Bott--Samelson varieties  in type $A_n$ considered in \cite{K18}.

In future, we plan to use the same method to complete \cite{K18}, and prove directly that Minkowski sums of FFLV polytopes yield Newton--Okounkov polytopes of Bott--Samelson varieties in type $A_n$.
In \cite{K18}, we used convex geometric Demazure operators in order to compare volumes, however, this was quite a roundabout way since FFLV polytopes themselves can not be constructed using Demazure operators.
On the other hand, push-pull operators seem to be a natural tool to construct FFLV polytopes and their Minkowski sums.

From a geometric viewpoint, points of the Bott--Samelson variety $R_5:=R_{(1,2,1,3,2)}$ can be identified with configurations $(a_1, l_1, a_2, \Pi, l_2)$ of subspaces in $\P^3$ such that $a_1$, $a_2$ are points, $l_1$, $l_2$ are lines, $\Pi$  is a plane, and
$$ a_1\in l_0, l_1; \ a_2\in l_1\subset \Pi_0; \ a_2\in l_2\subset \Pi,$$
where $l_0\subset \Pi_0$ are fixed line and plane in $\P^3$.
The intermediate Bott--Samelson varieties $R_1:=R_{(1)}$, $R_2:=R_{(1,2)}$, $R_3:=R_{(1,2,1)}$, $R_4:=R_{(1,2,1,3)}$  can be identified with the subvarieties of $R_5$, namely, define $R_4\subset R_5$ as the hypersurface given by the condition $\{l_2=l_1\}$.
Similarly, define $R_3\subset R_4$, $R_2\subset R_3$, and $R_1\subset R_2$ by the conditions $\Pi=\Pi_0$, $a_2=a_1$, and $l=l_0$, respectively.
Clearly, $R_i$ is the projectivization of a quotient tautological rank two bundle $E_{i-1}$ over $R_{i-1}$ for $i=1$,\ldots, $i=5$.
For instance, $R_5=\P(E_4)$ where the fiber of $E_4$ over a point $(a_1, l_1, a_2, \Pi, l_2)$ is the quotient $V(\Pi)/V(a_2)$ (by $V(P)$ we denote the vector subspace whose projectivization is the projective subspace $P$).
In particular, the Chow ring of $R_5$ is generated by the first Chern clasess
$\xi_1$,\ldots, $\xi_5$ of line bundles $\Oc_{E_1}(1)$,\ldots, $\Oc_{E_5}(1)$.
Note also that $R_1\simeq \P^1$, and $R_2$ is the blow up of $\P^2$ at one point.

We now construct a sequence of polytopes $\Delta_0:=\{\rm 0 \}$, $\Delta_1:=\Delta_{(1)}$, $\Delta_2:=\Delta_{(1,2)}$, $\Delta_3:=\Delta_{(1,2,1)}$, $\Delta_4:=\Delta_{(1,2,1,3)}$, $\Delta_5:=\Delta_{(1,2,1,3,2)}$ so that every polytope in the sequence is obtained as a push-pull polytope from the previous one.
Using the sequence of push-pull polytopes $\Delta_1$,\ldots, $\Delta_5$ and Theorem 4.1
we show that the Khovanskii--Pukhlikov ring of $\Delta_5$ is isomorphic to the Chow ring
of $R_5$.
More precisely, we construct families of analogous polytopes at each step.
In particular, we consider the family of segments
$\Delta_1(a)=\{u_1\in \R \ | \ 0\le u_1 \le a\}$ and
the family of trapezoids
$\Delta_2(a)=\{(u_1,u_2)\in \R^2 \ | \ 0\le u_1 \le a; \ 0\le u_2; \ u_1+u_2\le a+b \}$.

To construct $\Delta_3$ first note that the Chow ring of $R_3$ is isomorphic to
$$\Z[\xi_1,\xi_2,\xi_3]/(\xi_1^2, \ \xi_2^2-\xi_1\xi_2, \ \xi_3^2+\xi_2\xi_3+\xi_2^2)$$
(this is easy to compute by applying repeatedly projective bundle formula).
Hence, $CH^*(R_3)$ is isomorphic
to the Khovanskii--Pukhlikov ring of the push-pull polytope
$\Delta_3(a,b,c)\subset \R^3$ given by inequalities
$$0\le u_1\le a+b; \ 0\le u_3 \le c; \ 0\le u_2; \ u_2+u_3 \le b+c;
\ u_1+u_2+u_3 \le a+b+c.$$
The ring isomorphism takes $\xi_1$ to $\d_a$, $\xi_2$ to $\d_b$ and $\xi_3$ to
$\d_c-\d_b$.
Indeed, if we put $s=c$ in Example \ref{e.R_121} we get $\Delta_3(a,b,c)$.

It remains to construct $\Delta_4$ from $\Delta_3$ and $\Delta_5$ from $\Delta_4$.
Again by projective bundle formula we have the ring isomorphism
$$CH^*(R_5)\simeq CH^*(R_3)[\xi_4,\xi_5]/(\xi_4^2-\xi_2\xi_4, \
\xi_5^2-\xi_5(\xi_3+\xi_4-\xi_2)+\xi_3(\xi_4-\xi_2)).$$
Hence, if we take $\Delta_4=\Delta(P,Q)$, where $P=\Delta_3(a,b,c)$ and
$Q=\PQ=\Delta_3(a,b+1,c))$ we get a polytope whose Khovanskii--Pukhlikov ring
is isomorphic to $CH^*(R_4)$.
The ring isomorphism sends $\xi_4$ to $\d_d$.
There is a family of polytopes $\Delta_4(a,b,c,d)$ analogous to $\Delta_4$ defined
by inequalities:
$$0\le u_1, u_2, u_3, u_4; \ u_3 \le c; \  u_4\le d; \ u_1+u_4\le a+b+d; $$
$$\ u_2+u_3+u_4 \le b+c+d;  \  u_1+u_2+u_3+u_4 \le a+b+c+d.$$

To construct $\Delta_5$ as a push-pull polytope over $\Delta_4$ we first rewrite the
relation
$\xi_5^2-\xi_5(\xi_3+\xi_4-\xi_2)+\xi_3(\xi_4-\xi_2))$ (modulo the other relations) in
$CH^*(R_5)$ making the change of variable $x=\xi_4+\xi_5$.
We get the ring isomorphism
$$CH^*(R_5)=R_{\Delta_4}[x]/(x^2-(\d_c+\d_d)x+(\d_c^2+\d_d^2)).$$
This gives us a hint on how to choose $\PQ$ and $\Fc$ in order to construct $\Delta_5$ as
$\Delta(P,Q)$ for $P=\Delta_4(a,b,c,d)$.
Namely, $\PQ=\Delta(a,b,c+1,d+1)$, and
$$\Fc=\{\G_{124}\cap\G^0_4, \
\G_{234}\cap\G^0_3, \ \G_{234}\cap\G^0_4, \ \G_{1234}\cap\G^0_3, \
\G_{1234}\cap\G^0_4\},$$
 where
$\G^0_{i}$ denotes the facet $\{u_i=0\}$, and $\G_{124}$, $\G_{234}$ and $\G_{1234}$
denote the facets $\{u_1+u_2+u_4=a+b+d\}$, $\{u_2+u_3+u_4=b+c+d\}$ and
$\{u_1+u_2+u_3+u_4=a+b+c+d\}$, respectively.
It is not hard to check that $\d_{\PQ-P}=\d_c+\d_d$ and $D_{\Fc}=\d_c^2+\d_d^2$.
Applying Theorem \ref{t.main} we get that the ring $CH^*(R_5)$ is isomorphic to the
Pukhlikov--Khovanskii ring of the polytope $\Delta_5(a,b,c,d,e)\subset\R^5$ given by
$16$ inequalities:
$$u_1,u_2,u_3,u_4,u_5\ge0;\   u_1\le a+b+d; \ u_5\le e; \ u_3+u_5\le c+e; \  u_4+u_5\le d+e; $$
$$\ u_1+u_4+u_5\le a+b+d+e; \ u_2+u_3+u_5\le b+c+d+e; \ u_2+u_4+u_5\le b+c+d+e; $$
$$u_1+u_2+u_3+u_5\le a+b+c+d+e; \ u_1+u_2+u_4+u_5\le a+b+c+d+e; $$
$$u_2+u_3+u_4+2u_5\le b+c+d+2e; \ u_1+u_2+u_3+u_4+2u_5\le a+b+c+d+2e.$$
The isomorphism takes $\xi_5$ to $\d_e-\d_d$.

\def\rddots{\cdot^{\cdot^{\cdot}}}

It is easy to check that the polytope $\Delta_5$ is equal to the Minkowski sum of
FFLV polytopes $P_1(0,a)$, $P_2(0,b,b+c)$, $P_3(0,d,d+e,d+e)$ where the FFLV polytope
$P_{n-1}(\l_1,\ldots,\l_n)\subset \R^{\frac{n(n-1)}{2}}$ is defined by the inequalities
$u_k\ge 0$ and
$$\sum_{k\in D}u_k\le \l_j-\l_i$$
for all Dyck paths $D$ going from $\l_i$ to $\l_j$ in table $(FFLV)$  where $1\le i<j\le n$:
$$
\begin{array}{cccccccccc}
\ldots&       & \l_4    &       &\l_3   &       &\l_2   &    &\l_1 \\
      &\ldots &         &u_6    &       & u_3   &       &u_1 &     \\
      &       &\ldots   &       &u_5    &       & u_2   &    &     \\
      &       &         &\ldots &       & u_4   &       &    &     \\
      &       &         &       &\rddots &       &       &    &     \\
\end{array}.
\eqno{(FFLV)}$$
As a corollary, we get the following
\begin{prop}
The volume polynomial of the Minkowski sum $P_1(0,a)+P_2(0,b,b+c)+ P_2(0,d,d+e,d+e)$ of
FFLV polytopes  (considered as a polynomial in $a$, $b$, $c$, $d$ and $e$) is equal to
the self-intersection index of the divisor $a\xi_1+(b+c)\xi_2+c\xi_3+(d+e)\xi_4+e\xi_5$
on $R_{12132}$.
\end{prop}

\end{document}